\documentclass[a4paper,12pt, legno]{article}
\usepackage{amsmath}
\usepackage{amsfonts}
\usepackage{amsthm}
\usepackage{mathrsfs}
\usepackage[T1]{fontenc}
\usepackage{amsmath}
\usepackage{amsfonts}
\usepackage{amsthm}
\usepackage{mathrsfs}
\usepackage{amssymb}
\usepackage{pst-node}
\usepackage{tikz-cd} 

\newtheorem{thm}{Theorem}

\newtheorem{lem}[]{Lemma}

\newtheorem{cor}{Corollary}

\newtheorem{rem}[]{Remark}

\newcommand{\Rset}{\mathbb{R}}
\newcommand{\Cset}{\mathbb{C}}

\begin{document}

\title{On a  relation between classical and  free  infinitely divisible transforms}

\author{Zbigniew J.  Jurek  (University of Wroc\l aw)}

\date{June 28, 2017}

\maketitle
\begin{quote} \textbf{Abstract.} We study two ways  (levels) of finding free-probability  analogues of classical infinite divisible measures. More precisely,  we identify their  Voiculescu transforms.  For  free-selfdecomposable measures  we found the formula (a differential equation) for  their background driving transforms.
We illustrate our methods  on the  hyperbolic characteristic functions. As a by-product our approach potentially  may produce new formulas for definite integrals.

\emph{Mathematics Subject Classifications}(2010): Primary 60E07,
60H05, 60Z11; Secondary 44A05,  60B10.

\medskip

\medskip
\emph{Key words and phrases:}  infinite divisibility; free-infinite divisibilty; convolution semigroups; characteristic function; Voiculescu transform;
L\'evy-Khintchine formulas; L\'evy (spectral) measure;  Riemann zeta functions;  Euler function; digamma function.

\medskip
\medskip
\emph{Abbreviated title:  Classical and free-infinite divisibility}

\end{quote}
\maketitle

Addresses:

\medskip
\noindent Institute of Mathematics \\ University of Wroc\l aw  \\
Pl. Grunwaldzki 2/4
\\ 50-384 Wroc\l aw \\ Poland \\ www.math.uni.wroc.pl/$\sim$zjjurek ; \ \
e-mail: zjjurek@math.uni.wroc.pl

\newpage
There are many notions of infinite divisibility that exhibit some similarities and as well some differences. Here we study the classical infinite divisibility with respect to the convolution  $ \ast $ and the free-infinite divisibility for the box-plus  $\boxplus$ operation (Theorem 1). We  introduced free-probability analogues of the Laplace (double exponential) and the hyperbolic distributions (on the real line).  For the free-selfdecomposable Voiculescu transfroms we found an ordinary differential equation for their background driving transforms (Theorem 2).

The hyperbolic  distributions were  studied from the infinite divisibility point of view by Pitman and Yor (2003); cf. [18 ] . While here we have utilized the fact that all of them are in the proper subclass of selfdecomposable distributions (also called the class L distributions); Jurek (1996), cf.  [9]

\medskip
 The program of  the  study can be  viewed as a  particular case of  the following abstract set-up:  there are two abstract semigroups $ (\mathcal{S}_1, \circ)$ and $(\mathcal{S}_2, \diamond)$, two $1-1$ and \emph{onto} operators $A$ and $Z$ acting on domains $\mathcal{D}_1$ and $\mathcal{D}_2$, respectively  and  $1-1$ and \emph{onto} mapping $j$  between the domains $\mathcal{D}_1$ and $\mathcal{D}_2$. That is we have
\[
j : \mathcal{D}_1\to \mathcal{D}_2, \ \ \ \ \   A:\mathcal{D}_1\to \mathcal{S}_1 \ \ \ \  \ \ \mbox{and} \ \ \   \ \ \ \ Z:\mathcal{D}_2\to \mathcal{S}_2 .
\]
Consequently, the  diagram

  \[\begin{tikzcd}
 \mathcal{D}_1\arrow{r}{A}  \arrow[swap]{d}{j} &  (\mathcal{S}_1, \circ)  \arrow[swap]{d}{r} {} \\%
	\mathcal{D}_2 \arrow{r}{Z}& (\mathcal{S}_2, \diamond)
\end{tikzcd}
\]
 allows us to define the identification $r$ between  $ (\mathcal{S}_1, \circ)$ and $(\mathcal{S}_2, \diamond)$.

 Namely,  we say that
$\tilde{s} \in (\mathcal{S}_2, \diamond)$  is $\diamond$-\emph{analog}  or $\diamond$-\emph{counterpart}
of an $s\in (\mathcal{S}_1,\circ)$,  if  there exists
$x \in \mathcal{D}_1$ such that  $A(x)=s$,   $j(x)=y$ and  $Z(y)=\tilde{s}$. That is, we have  $r(A(x))=r(s):=\tilde{s}$, or $Z(j(x))=\tilde{s}.$

Similarly, $s\in(S_1,\circ)$ is $\circ$-analog of
 $\tilde{s}\in (\mathcal{S}_2,\diamond)$ if there exists $y\in \mathcal{D}_2$ such that $Z(y)=\tilde{s}, \ \  j^{-1}(y)=x$ and  $A(x)=s$.

\medskip
\medskip
\textbf{ 1. INFINITE DIVISIBILITY.}

\medskip
\textbf{1.1.} \  In  the  setting of this paper,  $(\mathcal{S}_1,\circ)\equiv (ID, \ast)$  is  the (classical)
convolution semigroup ID of all infinitely divisible probability measures 
$\mu$ on the real line with the convolution operation $\ast$. The characteristic functions $\phi$  (or the  Fourier transforms) are functions given as 
\[
\phi(t):=\int_{\Rset} e^{itx}\mu(dx), \   \  t  \in \Rset, \  \ \mbox{for some probability measure $\mu$}.
\]
Let $ \mathcal{D}_1:=\{\phi\in ID :  t\to (\phi)^{1/n}(t)  \mbox{is characteristic functions for} n=2,3,... \}$,
 that is,  $\mathcal{D}_1$ consists of all $\ast$ - infinitely divisible  characteristic functions. 
  
\noindent  Further,  let  $ \mathcal{D}_2:=\{[a,m]: a\in \Rset \ \mbox{and m is a finite Borel measure on $\Rset$,}\}$
 so it is a family of pairs [a,m].   Because of the following fundamental \emph{Khintchine representation formula}:
\begin{equation}
\phi \in \mathcal{D}_1\ \ \mbox{iff}  \ \ \phi(t)=\exp\{ ita +\int_{\Rset}(e^{itx}-1-\frac{itx}{1+x^2})\frac{1+x^2}{x^2}m(dx)  \}, \ t \in \Rset,
\end{equation}
for a uniquely determined parameters a  (a number) and m (a finite meaure), (for instance cf. [19], [17] or [1])  the mapping (an isomorphism) j given as
\begin{equation}
j:\mathcal{D}_1\to \mathcal{D}_2 \ \ \ \mbox{and} \ \ 
j (\phi):= [a,m] \ \ \mbox{iff}  \ \  \phi \ \mbox{is of the form  \  (1)},
\end{equation}
is well defined.

\medskip
\begin{rem}
\emph{(i) \  In some situations and applications instead of the finite measure m, in (1), one uses
a $\sigma$-finite measure \ \  $M(dx):=\frac{1+x^2}{x^2}m(dx)$ on $\Rset\setminus{\{0\}}$    \ \  (equivalently:   $m(dx):=\frac{x^2}{1+x^2}M(dx)$) and slightly changed the integrand as given below. Then the equality (1)  can be rewritten as follows:
\begin{equation*}
\phi(t)=\exp\{ itb- \frac{1}{2}t^2\sigma^2 +\int_{\Rset\setminus\{{0}\}}(e^{itx}-1- itx1_{\{|x|\le1\}}(x))M(dx) \}  \ \ \ (1a)
\end{equation*}
where
$\sigma^2:=m(\{0\})$ \ \  and \
  \  $b:=a+\int_{\Rset}x[1_{\{|x|\le 1\}}(x)-1/(1+x^2)]M(dx). $
\newline 
(ii)  \ The formula (1a) is called \emph{the L\'evy-Khintchine representation} of an infinitely divisible characteristic function (probability measures). The probability measure $\mu$ corresponding to (1a) is represented by the triple $\mu=[b, \sigma^2, M]$; cf. Parthasarathy (1967), Chapter VI;  or [1], [17].
\newline
(iii) \ The measure M has the  following stochastic meaning:  $M(A)$ is the expected number of jumps  that occur up to time  1 and are of sizes  in the set A,  of the corresponding L\'evy process $(Y(t), t\ge0)$, where $\phi$ is the characteristic function of  the random variable Y(1).}
\end{rem}

\medskip
\textbf{1.2.}  \  \  Further,  $(\mathcal{S}_2,\diamond)\equiv (ID, \boxplus)$ is the semigroup of all  $\boxplus$ free-infinitely divisible probability measures.
Namely for a probability measure $\nu$ on $\Rset$, one introduces its Voiculescu transform $V_{\nu}$ (an analogue of a characteristic function $\phi$) and an operation $\boxplus$ on measures in a such way that
\[
V_{\mu \boxplus \nu}(z)= V_{\mu}(z) +V_{\nu}(z);
\]
cf. Voiculescu (1999), cf.[20]. This in turn allows to introduce the notion of $\boxplus$-infinite divisibility and one has the following an analogue of the Khintchine  representations:

\begin{equation}
\nu \in (ID,\boxplus) \ \ \  \mbox{iff} \ \  \ V_{\nu}(z)= a + \int_{\Rset}\frac{1+zx}{z-x}m(dx), \ \ z\in \Cset \setminus \Rset,
\end{equation}
for uniquely determined a constant $a\in\Rset$ and finite (Borel) measure $m$; cf. [20] or [2], [3], [4].

\noindent However,  for the  uniqueness questions,  of  the representation (3),  is enough to consider Voiculescu transforms  only on the imaginary axis; Jurek (2006) , cf. [11]; (also [12], [13] and [8]).

\medskip
The formulas (1) and (2) suggest to define the following mappings:
\begin{equation}
A:  \mathcal{D}_1\to (ID, \ast) \ \  \ \mbox{given as } \ \  \ A(\phi):=\mu \ \ \mbox{iff} \ \phi(t)=\int_{\Rset}e^{itx}\mu(dx) \ \ t \in\Rset;
\end{equation}
and then to define $ r: (ID,\ast)\to (ID,\boxplus)$  as
\begin{equation}
 r(\mu):=\tilde{\mu}\  \ \mbox{iff} \ \ V_{\tilde{\mu}}(it)= it ^2\int_0^\infty  \overline{\log \phi(s)} e^{-ts}ds, \ \  t>0.
\end{equation}
Consequently, by (5),  we get  the composition  $r \circ A:\mathcal{D}_1\to (ID,\boxplus)$.

\medskip
On the other hand, on the level Z, using the mapping j ( from ( 2) ) we define
\begin{equation}
Z: \mathcal{D}_2\to (ID,\boxplus) \  \ \ \mbox{as} \ \ \  Z([a,m]):=\nu \  \ \mbox{iff} \ \ V_{\nu}(it) = a + \int_{\Rset}\frac{1+itx}{it-x}m(dx);
\end{equation}
So we have the following question:
\begin{equation}
\mbox{ Does} \ j(\phi)=[a,m] \ \mbox{imply that} \ r(A(\phi))=Z(j(\phi)) \ ? \ \ \  \mbox{That is,} \ \tilde{\mu}=\nu \ ?
\end{equation}
\begin{rem}
\emph{The idea of inserting  the same parameters a and m  into two different integral kernels (1) and (3) is due to  Bercovici - Pata (1999), Section 3; cf. [4] . [In our notation, it is the mapping on the level Z.]  
\newline
A different approach was proposed in Jurek (2006, 2007), cf. [11] and [12] and more recently repeated in Jurek (2016), cf. [16]. The key in those papers was the technique of the   random integral representation.  Here it is the mapping on the level A.}
\end{rem}

\medskip
\medskip\textbf{1.3.} \ \ Below we give straightforward  connections between the formulas (1) and (3) and prove the equality (7).
As in previous papers Jurek (2006), cf. [11] (or [12], [13]) we consider the transforms $V_{\nu}$ only on the imaginary line.

\begin{thm}
For each classical infinitely divisible $\mu \in (ID,\ast)$  with the  characteristic function $\phi_{\mu}$ there exists its unique free-infinitely divisible an analogue measure $\tilde{\mu}$ ($\tilde{\mu}\in (ID,\boxplus)$) such that its Voiculescu transform $V_{\tilde{\mu}}$ is given as
\begin{equation*}
 \ \ V_{\tilde{\mu}}(it)= it^2\,\int_0^{\infty} \overline{\log \phi_{\mu}(s)}\,e^{-ts}ds, \ \  t>0.  \ \ \  (A)
\end{equation*}
Furthermore, if $\mu$ has representation $\mu=[a,m]$ (in the Khintchine formula) then
\begin{equation*}
\ V_{\tilde{\mu}}(it) =a+\int_{\Rset}\frac{1+itx}{it-x}m(dx),  \ \ \ t>0.  \ \ \   (Z)
\end{equation*}
For a symmetric $\mu=[0,m]$ (real $\phi_{\mu}$)  we have
\begin{equation*}
 V_{\tilde{\mu}}(it)= it \int_0^{\infty}\int_{\Rset}( \cos(sx)-1)\frac{1+x^2}{x^2}m(dx) \,t\, e^{-st}ds= - it\,\int_{\Rset}\frac{1+x^2}{t^2 +x^2}m(dx)
\end{equation*}
\end{thm}
\begin{proof}
The fact that the function given in (A),  indeed, defines Voiculescu transform of an free-infinitely divisible measures was already shown in Jurek (2007), Corollary 6; cf. [12] (also  repeated in [13]). 

On the other hand, formula (Z) is obviously a Voiculescu transform of a free-infinitely measure in view of the characterization (3) above.

In order to show that both ways we get the same measure we show that
\begin{equation}
it^2\,\int_0^{\infty} \overline{\log \phi_{[a,m]}(s)}\,e^{-ts}ds=a+\int_{\Rset}\frac{1+itx}{it-x}m(dx),  \ \ \ t>0.  
\end{equation}
Taking the L\'evy exponent, as it is given  in Khintchine formula (1), computing the Laplace transform of the shift part $ita$ and then interchanging the order of integration  we have that
\begin{multline*}
LHS =  it^2\Big(  - \frac{i a}{t^2}  +\int_{\Rset}\frac{1+x^2}{x^2}\,[\int_0^{\infty}(e^{-isx}-1+\frac{isx}{1+x^2} )\,e^{-st}ds\,]\,m(dx) \Big)\\
= it^2\Big(  - \frac{i a}{t^2}  +\int_{\Rset}\frac{1+x^2}{x^2}\,[\,\frac{1}{ix+t}-\frac{1}{t}+\frac{ix}{1+x^2}\frac{1}{t^2}\,]\, m(dx)\Big)\\
= a  +\int_{\Rset}\frac{1+x^2}{x^2}\,it^2\,[\,\frac{-ix}{t (ix+t)}+\frac{ix}{(1+x^2)t^2}\,]\, m(dx)\Big)\\= a  +\int_{\Rset}\frac{1+x^2}{x^2}\,[\,\frac{tx}{ix+t}-\frac{x}{1+x^2}\,]\, m(dx)\Big)\\= a +\int_{\Rset}\frac{tx-i}{ix+t}m(dx) 
=a +\int_{\Rset}\frac{1+ itx}{it-x}m(dx)=RHS,
\end{multline*}
which concludes proof of the identity (8) and the first part of Theorem 1.

For the seconf part of the Theorem, we calculate as above, that is, we change the order of integration, utilize the fact that $m$ is symmetric measure and use the Laplace transform
\[
\int_0^{\infty}\cos(as) e^{-st}ds=\frac{t}{t^2 +a^2}.
\]
This conclude the argument for the second equality in Theorem 1.
\end{proof}

\begin{cor}
Let $\mathcal{E}_t, t>0,$ denotes  the exponential random variable  with parameter $t$ and the probability density $ t \,e^{-t x}1_{(0,\infty))}(x)$. Then
\[
\mathbb{E}[\log \phi_{\mu}(-\mathcal{E}_t)]=\int_0^{\infty} \overline{\log(\phi_{\mu}(s))}\,\,(t e^{-ts})\,ds= (it)^{-1}\,V_{\tilde{\mu}}(it), \ \ \mbox{for} \ \  t >0.
\]
Furthermore, if $\mu=[0,m]$ is a symetric (i.e., $\phi_{\mu}$ is real)  and m is a probability distribution of a random variable X that is stochastically independent of $\mathcal{E}_t$ then
\begin{equation}
\mathbb{E}\big[(1-\cos (\mathcal{E}_t\,X)\,)\,\frac{1+X^2}{X^2} \big]=\mathbb{E}\big[ \frac{1+X^2}{t^2+X^2}  \big], \ \ \mbox{for} \ \ t>0.
\end{equation}
\end{cor}
It follows from  the  second indentity in Theorem 1.

\medskip
\textbf{1.4.  Illustration of the  application of Theorem 1.} 
Let $C$ stands for \emph{hyperbolic-cosh} variable or its probability distribution. Then it is $\ast$-infinitely divisible and its characteristic function is equal  $\phi_C (t)= (\cosh t)^{-1}$.
From Theorem 1,  on the level of the characteristic functions (the mapping A),  its 
free-infinitely divisible analog $\tilde{C}$ (of the hyperbolic-cosh) has the following Voiculescu transform:
\begin{equation}
V_{\tilde{C}}(it)=- it^2\int_0^\infty \log \cosh(s)\, e^{-ts}ds=i[1-t\beta(t/2)], \  t>0,
\end{equation}
where $\beta$ is a special function defined in (19), formula (vii). Equality (10) is shown in Section 3.1 below. 

On the other, on the level of the parameters [a,m] (the mapping Z),  the hyperbolic-cosh has $a=0$ and
$m(dx)= \,\frac{1}{2}\,\frac{|x|}{1+x^2}\,\frac{1}{\sinh (\pi |x|/2)}\,dx  $. Consequently, by (8),  its $\boxplus$-free infinitely divisible
analog $\tilde{C}$ has the following Voiculescu transform
\begin{equation}
V_{\tilde{C}}(it)=- it  \int_0^\infty \frac{|x|}{t^2+x^2}\, \frac{1}{\sinh (\pi|x|/2)} dx=i \,[t\,\beta(t/2+1)- 1],
\end{equation}
for computational details see Section 3.1 below. 

As a byproduct  of (10) and (11)  we get the following functional  relation for the special function $\beta$:
\begin{equation}
\beta(s)+ \beta(s+1)= 1/s,  s>0,
\end{equation}
However, the formula (12) can also be obtained from other known representation of the special function $\beta$; cf. Section 3.0, formula (ix).

\begin{rem}
 \emph{Our two ways (two levels, two mappings) of getting free-infinitely divisible analogues of classical infinitely divisible characteristic function  may produce new unknown before explicite relation between some special functions.}
\end{rem}

\medskip
\textbf{2. SELFDECOMPOSABILITY.}

\medskip
\textbf{2.1.} \ An important and  a proper subclass of the class $( ID,\ast)$,  of all infinitely divisible measures,  is  \emph{the  class L},  also known
as the class of \emph{selfdecomposable probability measures}; cf. Jurek-Vervaat (1983), cf. [15] or Jurek and Mason (1993), cf. [14], Chapter 3.

 Let us recall  that the class L contains, among others, all stable probability measures, exponential distributions, t-Student distribution, chi-square, gamma,  Laplace,  hyperbolic-sine and hyperbolic-cosine measures (characteristic functions), etc; cf. Jurek (1997), cf. [10]. 

\medskip
For the purposes of this paper let us recall that for $\mu\in L$ (or equivalently for characteristic  function $\phi \in L) $ there exists an unique $\nu\in ID_{\log}$, \, i.e.,  infinitely divisible measures with finite logarithmic moment (or equivalently there exists an unique $\psi\in ID_{\log}$ )
such that 
\begin{equation}
\log \psi(t)= t \frac{d}{dt}\,\log \phi(t); \  \ \ \mbox{equivalently} \ \ 
\log\phi(t)=\int_0^t\log\psi(s)\frac{ds}{s}
\end{equation}
The above relations follow from \emph{the random integral representation of a selfdecomposable distributions}:  for each distribution $\mu \in L$ there exists a unique L\'evy process $Y_{\nu}$  such that
\begin{equation}
\mu=\mathcal{L}\big(\int_0^{\infty}e^{-s}dY_{\nu}(s)\big),  \ \  Y_{\nu}(s), s\ge 0, \ \ \ \mathcal{L}(Y_{\nu}(1))=\nu\in ID_{\log};
\end{equation}
cf. Jurek and Mason (1993), cf. [14] Theorems 3.4.6,  3.6.8 and Remark 3.6.9(4) or Jurek and Vervaat (1983), cf. [15]. 

The characteristic function $\psi$  (in (13)) is referred to as \emph{the background  driving characteristic function} (BDCF) of $\phi \in L$ and $Y_{\nu}$  \emph{the background driving L\'evy process} (BDLP) of $\mu$. 
\begin{rem}
\emph{A very nice argument, based on the random integral representation (14), for the existence of densities for all real-valued selfdecomposable variables is due to Jacod (1985), cf. [7]. His proof  is repeated  in Jurek (1997), cf.[10], pp.104-105.}
\end{rem}

\medskip
\noindent \textbf{2.2.} 
Here is a technical property (limit at infinity) of L\'evy exponent $\Phi$  of an infinitely divisible characteristic function  with a real parameter  $a$ and a finite measure $m$, that is,
\begin{multline}
 \Phi(t):=  ita + \int_{\Rset}(e^{i tx}-1- \frac{i tx}{1+x^2})\frac{1+x^2}{x^2}m(dx) \\   \qquad \qquad \qquad \qquad =
itb +  \int_{\Rset}(e^{i tx}-1-  itx 1_{|x|\le 1}(x))\frac{1+x^2}{x^2}m(dx), 
\end{multline} 
where  $ b:=a+\int_{\Rset}x[1_{\{|x|\le 1\}}(x)-1/(1+x^2)]\frac{1+x^2}{x^2}m(dx)$ and the finiteness  of the measure m guarantees the existence of the integral.

\begin{lem}
For any constants $c_1>0$ and $c_2>0$  and any L\'evy exponent $\Phi$ we have
\begin{equation*}
\lim_{t\to \infty} t^{c_1}e^{- c_2 t}\,\Phi(t)=   \lim_{t\to \infty} t^{c_1}e^{- c_2 t}\int_{\Rset}(e^{itx}-1- itx 1_{|x|\le 1}(x))\frac{1+x^2}{x^2}m(dx)=0.
\end{equation*}
\end{lem}
\begin{proof}
For pure degenerate $\Phi$, i.e., when $m=0$  in (15), Lemma 1 is obvious.  

Let us assume that  $b=0$. Since
\begin{multline*} 
|e^{itx}-1- itx|\le \min(\frac{|tx|^2}{2}, 2|tx|), \ \ \mbox{and} \ \  
|e^{itx}-1|\le \min( |tx|, 2)\le 2; 
\end{multline*}
(for instance, Billingsley (1986), cf. [5], pp. 352 and 353) therefore from (15) we get
\begin{multline*}
 t^{c_1} e^{-c_2 t}\, |\Phi(t)|  \le  t^{c_1}e^{- c_2 t} 
\int_{|x| \le 1}\frac{t^2}{2} x^2\frac{1+x^2}{x^2}m(dx) +    
 t^{c_1} e^{- c_2 t} \int_{|x|>1} 2\frac{1+x^2}{x^2}m(dx)  \\ \le \frac{1}{2} t^{c_1+2}e^{- c_2 t}\int_{|x| \le 1} (1+x^2) m(dx)  +  2 t^{c_1} e^{-c_2 t} \int_{|x|>1}(1+x^{-2})m(dx) \to 0 \ \mbox{as} \ t \to \infty,
\end{multline*}
which completes a proof of the lemma.
\end{proof}

Here, in Theorem 2, we have a free-selfdecomposability analogue of the differential relations (13), for the background driving equation, known for the classical selfdecomposability.
\begin{thm}
Let $\tilde{\phi}$ and $\tilde{\psi}$ be free-analogues of a selfdecomposable characteristic  function $\phi$ and its background driving characteristic function $\psi$, respectively.  Then  their Voiculescu transforms $V_{\tilde{\phi}}$ and $V_{\tilde{\psi}}$ satisfy the differential equation:
\begin{equation}
V_{\tilde{\psi}}(it) = V_{\tilde{\phi}}(it) -t \frac{d}{dt}[V_{\tilde{\phi}}(it)], \ \  t>0.
\end{equation}
Equivalently, in terms of $V_{\tilde{\psi}}$, we get
\begin{equation}
V_{\tilde{\phi}}(it)\,-t\,  V_{\tilde{\phi}}(i)=  - t  \int_1^t s^{-2}V_{\tilde{\psi}}(is)ds = t  \int_1^t V_{\tilde{\psi}}(is)\,d(s^{-1})  ,  \ \ t>0.
\end{equation}
\end{thm}
\begin{proof}
Note that using the definion (A) from Theorem 1, the relation (13) for classical selfdecomposabity and then Lemma 1 we have
\begin{multline*}
V_{\tilde{\psi}}(it) :=it^2\int_0^{\infty}\log \psi(-v)  e^{-t v}dv = it^2\int_0^{\infty}(\log \phi (-v))^{\prime} (-v) \,e^{-t v}dv  \\=
 it^2\int_0^{\infty}(\log \phi)^{\prime}(-v) (-1)\, (-v) \,e^{-t v}dv=   it^2\int_0^{\infty}(\log \phi)^{\prime}(-v) \,v e^{-t v}dv\\= it^2\big[\log \phi(-v)) v\,e^{-t v}\,\,|_{v=0}^{v=\infty} -\int_0^{\infty}\log \phi(-v)\,(1-t v)\,e^{-t v}dv\big]\\
= it^2[\int_0^{\infty} - \log \phi(-v)e^{-tv}dv+t\int_0^{\infty}\log \phi(-v)\,ve^{-tv}dv] \\ = -V_{\tilde{\phi}}(it)- it^3 \frac{d}{dt}[\int_0^{\infty}\log \phi(-v)e^{-tv}dv] = -V_{\tilde{\phi}}(it)- it^3  \frac{d}{dt}[ (it^2)^{-1}V_{\tilde{\phi}}(it)]\\=-V_{\tilde{\phi}}(it)- t^3  \frac{d}{dt}[  t^{-2}V_{\tilde{\phi}}(it)]= - V_{\tilde{\phi}}(it)-t^3[-2t^{-3}V_{\tilde{\phi}}(it)+t^{-2}\frac{d}{dt}V_{\tilde{\phi}}(it)]\\= V_{\tilde{\phi}}(it) - t\, \frac{d}{dt}V_{\tilde{\phi}}(it), \qquad \qquad 
\end{multline*}
which completes a proof of equality (16).

For the equality (17), note that (16) is a first-order linear differential equation that we can solve by the integrating factor method. More explicitly, note that  (16)  can be rewritten as follows
\[
t^{-1} V_{\tilde{\psi}}(it) = t^{-1} V_{\tilde{\phi}}(it) - \frac{d}{dt}[V_{\tilde{\phi}}(it)]= -t\,\frac{d}{dt} \,\big[ \frac{ V_{\tilde{\phi}}(it)}{t}\big]  
\] 
Hence, dividing by t and then integrating both sides over the interval [1,t] (or [t,1]) , we get
\[
 \frac{ V_{\tilde{\phi}}(it)}{t} - V_{\tilde{\phi}}(i) = - \int_1^t s^{-2}V_{\tilde{\psi}}(is)\,ds ,
\]
which completes the proof of the Theorem 2.
\end{proof}

\textbf{2.3. Illustration of the application of Theorem 2.} The hyperbolic-cosh function $\phi_C(t)=(\cosh t)^{-1}$ is selfdecmposable; Jurek (1996), cf. [9]. From (10) and (16) (in Theorem 2) we obtain $V_{\tilde{\psi_C}}$, free-probability analog of the background driving characteristic function $\psi_C$, as 
\begin{equation}
V_{\tilde{\psi _C}}(it)= i\,\big[ 1+\frac{1}{2}\,t^2\,\beta^\prime (\frac{1}{2}t )\,\big]=i\,\big[1+ \frac{t^2}{2}\zeta (2,\frac{t}{2})-\frac{t^2}{4}\zeta(2,\frac{t}{4}) \big], t>0.
\end{equation}
For the first equality one needs to put (10) into (16) and then use the formula that expresses $\beta^\prime$ in terms of Riemann's function $\zeta (2, a)$; for deatails  See Section 3.0.  [Two more ways of getting the above formula are  discussed in Section 4.]

\medskip
\textbf{3. FREE - PROBABILITY ANALOGUES OF THE  HYPERBOLIC CHARACTERISTIC FUNCTIONS.}

\medskip
\medskip
\textbf{3.0.} \ For an ease of reference we recall  definitions and basic properties of some  special function. All formulas followed by  a boldface reference number are taken  from  I. S. Gradshteyn, I. M. Ryzhik (1994), cf. [6].
\begin{multline}
(i)\ \ (a) \ \ \Gamma(z):=\int_0^{\infty}x^{z-1}e^{-x}dx, \ \  \ \Re z>0;  \ \ \ \mbox{(Euler function)}  \\  \qquad (b) \  \ \psi(z):=\frac{d}{dz}\ln \Gamma(z),   \  \ \ \Re z>0\ \  \mbox{ (digamma function)} \qquad \qquad \qquad  \\
(ii) \ \ \ \   \psi_n(z)\equiv \psi^{(n)}(z)=(-1)^{n+1}\,n! \,\,\zeta(n+1, z);  \ \ \  \textbf{8.363(8)},  \qquad \qquad \qquad \qquad \qquad \\  \mbox{( n-th derivative; called also as  polygamma)} . \\ 
(iii)  \ \   \ \ \psi(2z)=\frac{1}{2}(\psi(z)+\psi(z+\frac{1}{2})) +\ln 2 ; \ \ \  \       \textbf{8.365(6)}; \ \  \  \qquad \  \qquad \qquad \\
 (iv)\ \  \zeta(s,a):=\sum_{k=0}^{\infty}\frac{1}{(k+a)^s}, \ \Re s >1, \ -a \notin \mathbb{N},\ \mbox{(Riemann's zeta functions)};  \\   
(v) \ \  \ \zeta(s,a+1)=\zeta(s,a)-\frac{1}{a^s}; \ \ \ \ \  \zeta(s, a+1/2)=2^s\zeta(s,2a)-\zeta(s,a);\\ 
(vi) \ \ \  \zeta(2, t) - \frac{1}{4} \zeta(2, \frac{t}{2})= \frac{1}{4} \zeta(2, \frac{t+1}{2}); \ \ \  \mbox{(form (v))}; \qquad \qquad \qquad \qquad  \qquad \qquad \\
(vii) \    \ \ \beta(x):=\frac{1}{2}[\ \psi(\frac{x+1}{2})-\psi(\frac{x}{2})\ ], \ \ \beta(x) =\sum_{k=0}^{\infty}\frac{(-1)^k}{x+k},  \ \ -x \notin \mathbb{N}, \ \   \textbf{8.732(1)}.   \\
(viii)  \ \   \beta^\prime(x)= -\sum_{k=0}^\infty\frac{(-1)^k}{(x+k)^2}  =  \zeta (2, x)-\frac{1}{2}\,\zeta(2,\frac{x}{2}); \qquad   \ \ \  \textbf{8.374}; \qquad \qquad \qquad \\
(ix)  \  \ \ \  \beta(t)=\int_0^{\infty}\frac{1}{1+e^{-x}}\,e^{-tx}\,dx,  \ \ \ \Re t>0; \ \qquad  \textbf{8.371(2)} ;  \qquad \qquad \ \\
(x) \ \ \  ci(x)\equiv Ci(x):=- \int_x^\infty \frac{\cos u}{u}du;  \qquad \qquad \qquad \qquad  \qquad \qquad \qquad  \\ si(x):=-\int_x^\infty \frac{\sin u}{u}du=-\frac{\pi}{2} + Si(x), \ \mbox{where} \ \ Si(x):=\int_0^x \frac{\sin u}{u}du; \ \ 
\end{multline}

\medskip
\medskip
\textbf{3.1. Hyperbolic-cosine random variable.}

\medskip
Let $C$ stands for  \emph{the standard hyperbolic cosine variable}, that is,  the random variable with the characteristic  function
\begin{multline}
\phi_{C}(t):=\frac{1}{\cosh t}
=\exp\int_{\Rset}(\cos(tx)-1)\frac{1+x^2}{x^2}\,[\,\frac{1}{2}\,\frac{|x|}{1+x^2}\,\frac{1}{\sinh( \pi |x|/2)}\,]\,(dx)\\=\exp\int_{\Rset}(\cos(tx)-1)\,[\frac{1}{2}\frac{1}{|x| \sinh ( \pi |x|/2)}]\,(dx),
\end{multline}
where in the first  bracket [... ] is the density of the Khintchine finite measure $m_C$ corresponding to $\phi_C$ in the representation  (1)  and in the second one, is the density of  the sigma-finite L\'evy (spectral) measure M in (1a);  Jurek and Yor (2004), cf. [16].

\begin{cor}
The free-probability analog of the hyperbolic cosine characteristic function $\phi_C$ has the following Voiculescu transform
\begin{equation}
V_{\tilde{\phi_C}}(it)=i[1-t\beta(t/2)], t>0.
\end{equation}
As a by-product we infer the following identity for the function $\beta$:
\[
 \beta(s)+ \beta(s+ 1)= 1/s,  \ \ \  s>0.
\]
\end{cor}

\begin{proof}
(First proof.)
 
\noindent For first proof one needs the identity:
\begin{equation} 
\int_0^{\infty} e^{-\xi x}\ln (\cosh x)dx=
\frac{1}{\xi}\,[\,\beta(\xi /2)- 1/ \xi\,] , \ \Re \xi >0;
\mbox{cf. \textbf{4.342(2)}}\  in \, [6].
\end{equation}
Hence  and Theorem 1, level (A),
\begin{multline*}
V_{\tilde{\phi_C}}(it)=it^2\int_0^{\infty}\log \phi_C(-v)e^{- t v}dv=
 - it^2\int_0^{\infty}\log \cosh(v)e^{- t v}dv \\
= -it^2(\frac{1}{t}\,(\,\beta(t /2)- 1/t \,))= i \, [1- t \beta(t/2)], \ \ \ t>0;
\end{multline*}
which completes the calculation of (21).

(Second proof.) 

\noindent This time the integral identity needed below is as follows:
\begin{equation}
\int_0^{\infty} \frac{x\, dx}{(b^2+x^2)\ \sinh (\pi x)}= \frac{1}{2b} -\beta(b+1), \  b>0; \  \ \mbox{cf. \textbf{3.522(2)} \  in \, [6]}.
\end{equation}
 From first line in (20) we have that $\phi_C$ has finite Khintchine measure 
$$
m_C(dx)= \,\frac{1}{2}\,\frac{|x|}{1+x^2}\,\frac{1}{\sinh (\pi |x|/2)}\,dx $$
Consequently, from Theorem 1 we get 
\begin{multline*}
V_{\tilde{\phi_C}}(it)=-it\,\int_{\Rset}\frac{1+x^2}{t^2+x^2} \frac{1}{2}\,\frac{|x|}{1+x^2}\,\frac{1}{\sinh (\pi\,|x|/2)}dx \\
= -it\int_0^{\infty}\frac{x}{t^2+x^2}\frac{1}{\sinh (\pi x/2)}dx 
= -it\,\int_0^{\infty}\frac{y}{(t/2)^2+y^2}\frac{1}{\sinh \pi y}dy\\= -it [\frac{1}{t}-\beta(\frac{t}{2}+1)]
= i \,[t\,\beta(\frac{t}{2}+1)- 1] , \ \ \ \mbox{for}\ \  t>0.
\end{multline*}
From those two proofs we must have 
\begin{equation}
t\beta(\frac{t}{2}+1)-1= 1 -t\beta(\frac{t}{2}) \  \ \mbox{or} \ \ \beta(s)+ \beta(s+ 1)= 1/s;
\end{equation}
and this completes the proof of Corollary  2.
\end{proof}

\begin{rem} 
\emph{The identity (24) also follows from the fact that}

$\beta(t)=\int_0^{\infty}(1+e^{-x})^{-1}e^{-tx} dx$, t>0; 
\emph{ cf. \textbf{8.371(2)} in [6]} or Section 0, (ix).
\end{rem}

\medskip
\medskip
\textbf{3.2. Hyperbolic-sine variable.}

Let $S$ stands for  \emph{the standard hyperbolic-sine variable}, that is,  the random variable with the characteristic  function
\begin{multline}
\phi_{S}(t):=\frac{t}{\sinh t}
=\exp\int_{\Rset}(\cos(tx)-1)\frac{1+x^2}{x^2}\,[\,\frac{1}{2}\,\frac{|x|}{1+x^2}\,\frac{e^{-\pi|x|/2}}{\sinh (\pi |x|/2)}\,]\,dx\\=\exp\int_{\Rset}(\cos(tx)-1)\,[\frac{e^{-\pi |x|/2}}{2 |x| \sinh (\pi |x|/2)}]\,(dx),
\end{multline}
where in the first bracket $[...]$ is the density of the (Khintchine) finite measure $m_S$ corresponding to $\phi_S$ in (1)  and in the second one, is the density of  the sigma-finite L\'evy (spectral) measure $M_S$ in (1a).
\begin{cor}
The free-probability analog $\tilde{\phi_S}$ of the hyperbolic sine characteristic function $\phi_S$ has the following Voiculescu transform
\begin{equation}
V_{\tilde{\phi_S}}(it)=i[t\psi(t/2)-t \ln(t/2)+1], t>0.
\end{equation}
\end{cor}
\begin{proof}
(First proof.) The key integral identity for (26) is the following one:
$$
\int_0^{\infty} e^{-\xi x}(\ln (\sinh x) - \ln x)dx=
\frac{1}{\xi}[ \ln(\xi/2) - 1/\xi-\psi(\xi/2)] , \ \ \ \Re \xi >0; 
$$
cf. \textbf{4.342(3)} in [6]. 

\noindent (NOTE the misprint in [6]; comp.  $www.mathtable.com/errata/gr6_errata.pdf$).

Hence and from Theorem 1, equality (A), we get
\begin{multline}
V_{\tilde{\phi_S}}(it)= i t^2\int_0^{\infty}\log\phi_S(-v)e^{-tv}dv
= - it^2\,\int_0^{\infty}[\log \sinh v - \log v]\,e^{- tv}dv\\
= -it^2 \frac{1}{t}[ \ln(t/2) - 1/t-\psi(t/2)] 
= i [t \psi(\frac{t}{2}) - t \ln \frac{t}{2}+1] ,\ \ \ 
\end{multline}
which proves (26).

(Second proof.)  This time we need  the formula
$$
\int_0^{\infty}\frac{ x\,dx}{(x^2+\beta^2)(e^{\mu x}-1)}=\frac{1}{2}\big[\log(\frac{\beta\mu}{2\pi})-\frac{\pi}{\beta\mu}-\psi(\frac{\beta\mu}{2 \pi}) \big], \, \Re \beta>0, \ \Re \mu>0;
$$
cf. \ \textbf{3.415(1)}  in [6].

From first line in (25) we have that the Khintchine (finite) measure $m_S$ (for $\phi_S$)  is equal to
$$
m_S(dx)=\frac{1}{2}\,\frac{|x|}{1+x^2}\,\frac{e^{-\pi|x|/2}}{\sinh (\pi |x|/2)}\,dx=\frac{|x|}{1+x^2}\frac{1}{1-e^{-\pi |x|}}dx
$$
(also Jurek-Yor (2004), cf. [16].) Thus the above indentity and  Theorem 1, equality (Z), give 
\begin{multline}
V_{\tilde{\phi_S}}(it) 
=- 2it\,\int_0^{\infty}\frac{x}{t^2+x^2}\,\frac{1}{e^{\pi x}-1}\,dx= -it\big[\ln(\frac{t}{2})-\frac{1}{t}-\psi(\frac{t}{2})\big]\\ =i[t\,\psi(t/2)-t \log (t/2)+1],\ \ t>0,  \ \ \ \ \ 
\end{multline}
which coincides with  (26). This completes  a proof of Corollary 3.
\end{proof}

\medskip
\textbf{3.3. The hyperbolic-tangent variable.}

Let $T$ stands for  \emph{the standard hyperbolic-tangent variable}, that is,  the random variable with the characteristic  function $\phi_T(t)=\frac{\tanh t}{t}$.  It's Khintchine representation is as follows:
\begin{equation}
\phi_T(t)=\frac{\tanh t}{t}= \exp(\int_{-\infty}^{\infty}(\cos tx-1)\,[\,\frac{1}{2}\frac{|x|}{1+x^2}\frac{e^{-\pi |x|/4}}{\cosh ( \pi |x|/4)}\,]dx; 
\end{equation}
where in the bracket $[...]$ there is the density of the finite Khintchine measue $m_T$ from the formula (1).
\begin{cor}
The free-probability  analog $\tilde{\phi_T}$ of hyperbolic tangent characteristic function has the following Voiculescu transform
\begin{equation}
V_{\tilde{\phi_T}}(it)=  it\, [\,\ln(\frac{t}{2})-\beta(\frac{t}{2})-\psi(\frac{t}{2})\,] =it\,[ \ln (\frac{t}{4})- \psi(\frac{t}{4}+\frac{1}{2})],  \ \ \  t>0.
\end{equation}
Consequently, we get the identity for Euler's  function
\begin{equation}
2\,\psi(2s) - \psi (s)- \psi(s+1/2)= 2 \ln 2  \ \ \  s>0.
\end{equation}
\end{cor}
\begin{proof} (First proof.)

From the equality $\phi_C(t)=\phi_S(t)\cdot \phi_T(t)$, Theorem 1, (A),  and Corollaries 2 and 3  we get
\begin{multline*}
V_{\tilde{\phi_T}}(it)=it^2\int_0^{\infty}[\,\overline{\log\phi_C(t)}- \overline{\log\phi_S(t)}\,]e^{-ts}ds=V_{\tilde{\phi_C}}(it)-V_{\tilde{\phi_S}}(it) \\= 
 i[1-t\beta(t/2)]   -i[t\,\psi(t/2)-t \log (t/2)+1]=it [\ln(t/2)-\beta(t/2)-\psi(t/2)],
\end{multline*}
which  gives the first equality in (30).

(Second proof.) 

This time we need the formula \textbf{3.415(3)} in [6], that is, 
\[
\int_0^\infty\frac{x}{(x^2+\beta^2)(e^{\mu x}+1)}dx=\frac{1}{2}\Big[\psi(\frac{\beta\mu}{2\pi}+\frac{1}{2})-
\ln (\frac{\beta\mu}{2\pi})\Big], \ \Re\beta>0, \  \Re\mu>0.
\]
Since  $1-\tanh x= \frac{e^{-x}}{\cosh  x}=\frac{2}{e^{2x}+1} $,  using the above and Theorem 1, equality (Z), 
 we have
\begin{multline}
V_{\tilde{\phi_T}} (it)= - it\,\int_{\Rset}\frac{1+x^2}{t^2+x^2}\frac{1}{2}\frac{|x|}{1+x^2}\frac{e^{-\pi |x|/4}}{\cosh ( \pi |x|/4)}\,dx \\= 
- 2 it\int_0^\infty \frac{x}{(t^2+x^2)(e^{\pi x/2}+1)}\,dx = - 2it \frac{1}{2}\Big[\psi(\frac{t \pi/2}{2\pi}+\frac{1}{2})-
\ln (\frac{t \pi/2}{2\pi})\Big] \\
= it \Big[\ln (\frac{t}{4}) -  \psi(\frac{t}{4}+\frac{1}{2}) \Big], \qquad \qquad \qquad
\end{multline} 
that is the second equality in (30). Consequently, by Theorem 1, 

$
\ln (\frac{t}{4})-\psi(\frac{t}{4}+\frac{1}{2})=\ln(t/2)-\beta(t/2)-\psi(t/2),   \ \ \  t>0, $
or  equivalently,
\begin{equation}
\psi(2s) -1/2 \psi (s)- 1/2 \psi(s+1/2)= \ln 2  \ \ \  s>0.
\end{equation}
which completes the proof.
\end{proof}

\begin{rem}
\emph{ (a) \ Note that the above identity (33) concides with the formula \textbf{8.365(6)}  in [6], for $n=2$, in [6]. (cf. also Section 3.0, formula (iii)). 
 \newline  (b) \ By reasoning as in the seconds proofs of Corollaries 2 and 3, and using (32),  we get for $t>0$
\begin{equation*}
\int_0^{\infty}\frac{x}{t^2+x^2}(1-\tanh (\pi x))dx=\psi(2t)+\beta(2t)-\log(2t)= \psi(t+\frac{1}{2})-\log t . 
\end{equation*}
\newline (c) In particular, $\int_0^{\infty}\frac{x}{1+x^2}(1-\tanh (\pi x))dx=\psi(3/2)$.}
\end{rem}

\medskip
\medskip
\textbf{4. FREE - PROBABILITY ANALOGUES OF  BACKGROUND DRIVING FUNCTIONALS OF HYPERBOLIC DISTRIBUTIONS.} 

Since  three hyperbolic characteristic functions $\phi_C, \phi_S$ and $\phi_T$  (of the random variables C, S and T) are selfdecomosable (in oder words, in L\'evy class L) therefore they admit infinitey divisible background driving characteristic functions (BDCF) $\psi_C, \psi_S$ and $\psi_T$, respectively. Further, if $N_C$, $N_S$ and $N_T$ are their L\'evy spectral measures then
\begin{equation}
\psi_C(t)=\exp[- t\tanh t]\,;  \ N_C(dx)=\frac{\pi}{4}\,\frac{\cosh ((\pi x)/2)}{\sinh^2((\pi x)/2)} dx  \  \mbox{on}\  \Rset \setminus\{0\};
\end{equation}
\begin{equation}
\psi_S(t)=\exp[1-t\coth t]; \ \ N_S(dx)=\frac{\pi}{4}\,\frac{1}{\sinh^2((\pi x)/2)} dx \ \mbox{on}\  \Rset \setminus\{0\}
\end{equation}
\begin{equation}
\psi_T(t)=\exp\big[ \frac{2t}{\sinh (2t)}-1\big]; \  N_T(dx)=\frac{\pi}{8}\frac{1}{\cosh^2(\pi\,x/4)}\, dx \   \mbox{on}\  \Rset \setminus\{0\};
\end{equation}
 Jurek-Yor (2004), cf. [16].
 \begin{rem}
 \emph{(a) Note that $\psi_T$ is a characteristic function of a compound Poisson distribution.
\newline
 (b) Elementary calculations give  $\psi_C/\psi_S=\psi_T$.}  
\end{rem}
Let $\tilde{\psi_C}, \tilde{\psi_S} $ and  $\tilde{\psi_T}$ be the free-probability analogues of  BDCF for  $\psi_C, \psi_S, \psi_T$, respectively.  Note that in those cases we have three possible ways of finding them: two because of the  levels  A and Z  from Theorem 1 and, if possible, the third one by the differential equation (16) in Theorem 2.
\begin{cor}
Let $\tilde{\psi_C}, \tilde{\psi_S} $ and  $\tilde{\psi_T}$ be the free-analogues the corresponding BDCF. Then their Voiculescu transforms are as follows:
\begin{equation}
 (a) \, \ V_{\tilde{\psi_C}}(it)=i\,[\,t^2/2 \,\zeta(2, t/2)- t^2/4\zeta(2, t/4)+1\,]
\end{equation}
\begin{equation}
(b) \ \ V_{\tilde{\psi_S}}(it)= i[1+t- (1/2) \, t^2\zeta(2, \, t/2)].
\end{equation}
\begin{equation}
(c) \ \ V_{\tilde{\psi_T}}(it)=it[ t\zeta(2, t/2)-t/4 \zeta(2, t/4)-1]=it [t/4 \zeta(2, (t+2)/4)-1]
\end{equation}
\end{cor}
\begin{proof}
(a) From Theorem 1, using the  \emph{Mathematica} and formulas (v) ad (vi) from Section 3.0 we get
\begin{multline*}
V_{\tilde{\psi_C}}(it)=-it^2\int_0^{\infty}v \tanh v e^{-tv}dv
= -it^2\,[\,\frac{1}{8}\,\zeta(2,\frac{t}{4}+1)-\,\frac{1}{8}\,\zeta(2,\frac{t}{4}+\frac{1}{2})+\frac{1}{t^2}\,] \\
=-it^2[1/4 \zeta(2,t/4)-1/2 \zeta(2,t/2)-1/t^2]= i [t^2/2\zeta(2,t/2)-1/4t^2\zeta(2, t/4) + 1].
\end{multline*}
(Note that  equality (a), by a different way (Theorem 2), was already computed in  Illustration 2.3,  formula (18).)

\medskip
(b) From the equality  \textbf{3.551(3)} in [6]:
\[
\int_0^\infty x^{\mu-1}e^{-\beta x}\coth x dx= \Gamma(\mu)\,[2^{1-\mu}\zeta(\mu,\beta/2)-\beta^{-\mu}], \Re \mu>1, \Re \beta>0; 
\]
Putting  $\mu=2$ and $\beta=t$ we get
\begin{multline*}
V_{\tilde{\psi_S}}(it)=\,it^2\int_0^{\infty}(1-v \coth v) e^{-tv}dv= it^2[t^{-1}-\int_0^{\infty}v \coth v e^{-vt}dv]\\
= it^2[t^{-1}-\frac{1}{2}\zeta(2,\frac{t}{2})+\frac{1}{t^2}]=i[t-\frac{1}{2}t^2(\zeta(2;\frac{t}{2})+1]=\,i\,[1+t-\frac{1}{2}t^2\zeta(2;\frac{t}{2})\,];
\end{multline*}
that gives (b).

\medskip
(c)  By Remark  7(b),  $ \log \psi_T=\log \psi_C-\log \psi_S$.  Thus  using (a) and (b) we get
\begin{multline*}
V_{\tilde{\psi_T}}(it)= - it^2\int_0^\infty [\log \psi_C(v)-\log\psi_S(v)]ve^{-tv}dv \\
i\,\big[\,(t^2/2) \,\zeta(2, t/2)- (t^2/4)\,\zeta(2, t/4)+1\, \big] -  i[1+t- (1/2) \, t^2\zeta(2, \,(t/2)]\\
= i [ t^2\zeta(2, t/2) -t^2/4\zeta(2,t/4)-t ]=it[t\zeta(2, t/2)-t/4 \zeta(2,t/4)-1]\\
= it\,[ \frac{t}{4} \,\zeta(2, \frac{t+2}{4})-1], \qquad \qquad \qquad \qquad  
\end{multline*}
which gives (c).

(An alternative proof is possible by using the differential equation from Theorem 2 and the formulas in Corollary 5.)
\end{proof}

\medskip
\medskip
\textbf{5. \ FREE LAPLACE (or FREE DOUBLE EXPONENTIAL) MEASURE.}

All  the hyperbolic characteristic functions $\phi_C, \phi_S, \phi_T$, discussed in the previous sections,  are infinite products of the Laplace (called also double exponential) distributions; Jurek (1996), cf. [9]. Thus we include it in this paper as well.

 \textbf{5.1.} \  Recall that that \emph{the double exponential (2e) (or Laplace) distribution}  has the probability density  $f(x):=2^{-1}e^{-|x|}$, ($x \in \Rset$) and the characteristic function
\begin{multline}
\phi_{2e}(t)=\frac{1}{1+t^2}=\exp\int_{\Rset}(\cos tx-1) \frac{1+x^2}{x^2}[\,
\frac{e^{-|x|}|x|}{1+x^2}]dx \\=\exp\int_{\Rset\setminus \{0\}}(\cos tx-1)[\frac{e^{-|x|}}{|x|}]dx ,
\end{multline}
where in the first square bracket [...] there is the finite Khintchine spectral measure $m_{2e}$ and in the second one there is the L\'evy spectral measure $M_{2e}$. 

\medskip
\begin{cor}
The free analogue $\tilde{\phi_{2e}}$  of the double exponential distribution has the following Voiculescu transform
$$
V_{\Tilde{\phi_{2e}}}(it)=    2 it\,[\,  ci(t)\cos t + si(t)\sin t \,]    \ \ \  \big(= - 2it \int_0^{\infty}\frac{\cos w}{w+t}dw\big) ,  \ \  t>0.
$$ 
\end{cor}
\emph{Proof.} (First the argument via the equality (A) in Theorem 1 .)

For $\Re \beta >0$ and $\Re \xi>0$ we have the identity 
\[
\int_0^{\infty}e^{-\xi x}\ln(\beta^2+x^2)dx=\frac{2}{\xi}[\ln \beta -ci(\beta \xi)\cos(\beta \xi)-si(\beta\xi)\sin(\beta\xi)], 
\]
where $
si (x) : =  - \int_x^{\infty} \frac{\sin t}{t}dt;  \      ci (x): = - \int_x^{\infty}\frac{\cos t}{t}dt, 
$ are the integral-sine and integral-cosine functions, respectively; cf. \textbf{4.338(1)} in [6] or Section 0, formula (x).

Therefore 
\begin{multline}
V_{\tilde{\phi_{2e}}} (it)= it^2\int_0^{\infty}\overline{\log \phi_{2e}(v)}e^{-tv}dv  =  -it^2\int_0^{\infty}\log(1+v^2)e^{-tv}dv ] \ \  \mbox{(by 4.338(1))} \\
\   -it^2\,[2 t^{-1} (-ci(t)\cos t - si(t)\sin t)]= -2it (\cos t \int_t^{\infty}\frac{\cos x}{x}dx +\sin t \int_t^{\infty}\frac{\sin x}{ x}dx)\\
= -2it\,\,\int_t^{\infty}\frac{\cos t \cos x +\sin t\sin x}{x}\,dx= - 2it\int_t^{\infty}\frac{\cos (x-t)}{x}\,dx \ \ (w:=x-t) \\
 = - 2it \int_0^{\infty}\frac{\cos w}{w+t}dw ,  \ \ \  t>0.\qquad \qquad \qquad
\end{multline}

(Second argument via  formula (Z) in Theorem 1) 

By  second part of Theorem 1 and by \textbf{3.354(2)} in [6] we get
\begin{multline}
V_{\tilde{\phi_{2e}}}(it) = 
-it\,\int_{\Rset}\frac{1+x^2}{t^2+x^2}\,[\,\frac{|x|}{1+x^2}e^{-|x|}\,]\,dx= - 2it\int_0^{\infty}\frac{x}{t^2+x^2} e^{-x}dx= \\
= 2it \,[ ci (t)\cos t +si (t)\sin t]=
 -2it \int_0^{\infty}\frac{\cos w}{w+t}dw, \ \ \
\end{multline}
which coincides with previous calculations.

\medskip
\textbf{5.2.} The Laplace (or double exponential) characateristic function $\phi_{2e}=(1+t^2)^{-1}$ is selfdecomposable.  Therefore  it has the background driving charactristic function $\psi_{2e}$ related to $\phi_{2e}$ via (13). Hence
\begin{multline}
\psi_{2e}(t)= \exp[t \frac{\phi_{2e}^{\prime}(t)}{\phi_{2e}(t)}]   =\exp(- \frac{2t^2}{1+t^2})= \exp 2(\frac{1}{1+t^2}-1 ) \  \mbox{(compund Poisson)} \\ =  \exp 2(\int_{\Rset}(e^{itx}-1)\frac{1}{2}e^{-|x|}dx= \exp \int_{\Rset} (e^{itx}-1- \frac{itx}{1+x^2})e^{-|x|}dx;
\end{multline}
from which we infer that $m_{\psi_{2e}}(dx):= \frac{x^2}{1+x^2} e^{-|x|} dx$  is its finite Khintchine measure in (1) . Here is the  Voiculescu transfom of the free analog of 
$\psi_{2e}$.
\begin{cor}
The free analog of BDCF $\psi_{2e}$ has Voiculescu transform $V_{\tilde{\psi_{2e}}}$  given as 
\begin{equation}
V_{\tilde{\psi_{2e}}}(it)= 2it \ [ \ t \big( ci(t)\sin(t)- si(t)\cos(t) \big)-1\ ],
\end{equation}
where $si$ and $ci$ are the integral sine and  cosine functions.
\end{cor}
\emph{Proof.} (First proof.)  From Corollary 6,  $V_{\Tilde{\phi_{2e}}}(it)= 2 it\, \alpha(t)$ where $\alpha(t):=  ci(t)\cos t + si(t)\sin t$. Then  Theorem 2 gives  that 
\begin{multline*}
V_{\tilde{\psi_{2e}}}(it)= V_ {\tilde{\phi_{2e}}}(it ) - t \frac{d}{dt} [V_ {\tilde{\phi_{2e}}}(it)  ]   = 2 it \alpha (t)- t (2 i \alpha(t)+2it \alpha^{\prime}(t)) = - 2it^2 \alpha^{\prime}(t) \\ = -2it^2 (t^{-1} - ( (ci(t)\sin(t)-si(t)\cos(t)))=
2it[t( ci(t)\sin(t)-si(t)\cos(t))-1],
\end{multline*}
that completes the reasoning for the formula  (44).

(Second proof.) Here we will use the equality
$$
\int_0^\infty \frac{e^{-\xi x}}{\beta^2+x^2} dx= \frac{1}{\beta}[ci(\xi \beta)\sin (\xi\beta)-si(\xi\beta)\cos(\xi\beta)], \ \Re \xi >0, \ \Re \beta >0;
$$
cf. 3.354 (1) in [6 ] or  the identity used in the first proof of Corollary 6. Hence and  from Theorem 1 (level A) and (43)
\begin{multline*}
V_{\tilde{\psi_{2e}}}(it)= i t^2\int_0^\infty \log \psi_{2e}(s) e^{-t s}ds=  2 i t^2 \int_0^\infty [\frac{1}{1+s^2} - 1]e^{-st}ds \\ =
 2it^2 \Big(  [ci( t)\sin (t)-si(t)\cos(t)] - t^{-1} \Big) = 2it \ [ \ t \big( ci(t)\sin(t)- si(t)\cos(t) \big)-1\ ],
\end{multline*}
which coincides with (44).

(Third proof.) Now we use the level Z from Theorem 1 and the finite Khintchine measures 
$m_{\psi}(dx)= \frac{x^2}{1+x^2} e^{-|x|} dx$  from repersentation (43). Thus
\begin{multline*}
V_{\tilde{\psi_{2e}}}(it)= - i t\int_{\Rset} \frac{x^2}{t^2+x^2} e^{-|x|}dx=
 - i t \int_{\Rset} (1- \frac{t^2}{t^2+x^2}) e^{-|x|}dx 
\\  = 2 it \Big(t^2\int_0^\infty \frac{1}{t^2+x^2}e^{-x}dx -1\Big)= 2 it \big(   t^2 [\frac{1}{t}(\big( ci(t)\sin(t)- si(t)\cos(t) ]-1\big) \\ =2it \ [ \ t \big( ci(t)\sin(t)- si(t)\cos(t) \big)-1\ ],
\end{multline*}
that completes the last argument in the proof of Corollary 7.

\medskip
\textbf{Acknowledgment.} I would like to thank K.Topolski, K. Makaro and R. Suwalski  from University of Wroclaw, for their technical help in setting up diagrams, and in using  Wolfram and Mathematica  programs for computing definite integrals.

\bigskip
\bigskip
\noindent {\bf References}

[1] A. Araujo and E. Gine (1980), \emph{The central limit theorem for real and Banach valued random variables,} J. Wiley , New York.

[2] O. Barndorff-Nielsen and  S. Thorbjorsen (2006), Classical and free infinite divisibility
and L\'evy process; in \emph{Lect. Notes in Math.} \textbf{1866}, pp. 33-159.

[3] H. Bercovici and D. V. Voiculescu (1993), Free convolution of measures with unbounded support, \emph{Indiana Univ. Math. J.}, vol.42, 733-773.

[4] H. Bercovici  and V. Pata (1999), Stable laws and domains of attraction in free probability theory; \emph{ Annals of Math.}, vol.\textbf{149}, pp. 1023-1060.

[5] P. Bilingsley (1986), \emph{Probability and mesaure}, John Wiley $\&$ Sons, New York (Second Edition).

[6]  I. S. Gradshteyn, I. M. Ryzhik (1994), Table of integrals, series and products, $5^{th}$ Edition, Academic Press, New York.

[7] J. Jacod (1985), Grossissement de filtratrion et processeusd"Ornstein-Uhlenbek generalise.  In : \emph{ Grossissement de filtartion: examples et applications}; Springer, D.Julien and M. Yor  Eds. Lect.  Notes in Math. 1118, pp. 37-44.

[8]  L. Jankowski and Z. J.  Jurek (2012), Remarks on restricted Nevalinna transforms, \emph{Demonstratio Math.} vol. XLV,  no 2,  pp. 297-307.

[9]  Z. J.  Jurek (1996),  Series of independent exponential random variables.   
In:  \emph{Proceedings of the Seventh Japan-Russia Symposium 
Probab. Theor. and Math. Stat.}, Tokyo 26 - 30 July 1995; World Scientific; pp. 174-182.

[10]  Z. J. Jurek (1997),  Selfdecomposability: an exeption or a rule ?, \emph{Annales Universitatis Marie Curie-Sklodowska, Lublin-Polonia}, vol. 51 , Sectio A, pp. 93-107.

[11]  Z. J.  Jurek (2006), Cauchy transforms of measures as some functionals of Fourier transforms, \emph{Probab. Math. Stat.} vol. \textbf{26} , Fasc. 1, pp. 187-200.

[12] Z. J. Jurek (2007), Random integral representations for free-infinitely divisible and tempered stable distributions, \emph{ Stat. Probab. Letters} vol. 77, pp. 417-425.

[13]  Z. J . Jurek (2016), On a method of introducing  free-infinitely divisible probability measures, \emph{Demonstratio Math.} vol. 49, No 2, pp. 236-251.

[14] Z. J. Jurek and J. D. Mason (1993), \emph{Opertor-limit distributions in probability theory}, J. Wiley and Sons, New York.

[15] Z. J. Jurek and W. Vervaat (1983),  An integral represenation for selfdecopmposable  Banach space valued random variables, \emph{Z. Wahrscheinlichkeitstheorie verw. Gebiete}, vol. 62, pp. 247-262.

[16] Z. J. Jurek and M. Yor (2004), Selfdecomposable laws  associated with hyperbolic functions,  \emph{Probab. Math. Stat.}, vol. 24, Fasc. 1, pp. 181-191.

[17] M. Meerscheart, P. Scheffler (2001), \emph{Limit distributions for sums of independent random vectors},  Wiley Series in Probability and Statistics, John Wiley $\&$ Sons, New York.

[18] J. Pitman and M. Yor (2003), Infinitely divisible laws associated with hyperbolic functions, \emph{Cand. J. Math.} vol. 55 (2), pp. 292-330. 

[19] K. R. Parthasarathy (1967), \emph{Probability measures on metric spaces}, Academic Press, New York and London.

[20]  D. Voiculescu (1999),  Lectures on free probability. In: \emph{Lectures on probability theory and statistics}, Saint-Flour XXVIII - 1998;  pp. 279-349. Springer

\end{document}